\documentclass[leqno,english,11pt]{amsart}
\usepackage[active]{srcltx}
\usepackage{color}

\usepackage{epsfig}
\usepackage{amssymb}
\usepackage{latexsym}
\usepackage[latin1]{inputenc}
\usepackage{amsmath,amsthm,amssymb,babel}
\newtheorem{theorem}{Theorem}

\newtheorem{proposition}{Proposition}
\newtheorem{cor}{Corollary}

\newtheorem{remark}{Remark}

\def\endproof{\hfill $\Box$\par\vskip3mm}

\def\neweq#1{\begin{equation}\label{#1}}
\def\endeq{\end{equation}}

\def\phi{\varphi}

\def\PP{{\mathbb P} }

\def\di{\displaystyle}

\begin{document}

\title[Comparison results for the p-torsional rigidity]{\bf Comparison results for the $p$-torsional rigidity on convex domains}

\author[ C. Enache, M. Mih\u ailescu, and D. Stancu-Dumitru ]{Cristian Enache, Mihai Mih\u ailescu$\,^{\star}$,
and Denisa Stancu-Dumitru }

\address{Cristian Enache
\hfill\break\indent Department of Mathematics and Statistics
\hfill\break\indent American University of Sharjah 
\hfill\break\indent P.O. Box 26666, Sharjah, United Arab Emirates.
\hfill\break\indent {\tt cenache@aus.edu}}

\address{Mihai Mih\u ailescu
\hfill\break\indent Department of Mathematics
\hfill\break\indent University of Craiova
\hfill\break\indent 200585 Craiova, Romania
\hfill\break\indent and
\hfill\break\indent ''Gheorghe Mihoc - Caius Iacob" Institute of Mathematical Statistics and Applied Mathematics of the Romanian Academy
\hfill\break\indent 050711 Bucharest, Romania.
\hfill\break\indent {\tt mmihailes@yahoo.com}}

\address{Denisa Stancu-Dumitru
\hfill\break\indent Department of Mathematics and Computer Sciences
\hfill\break\indent National University of Science and Technology  Politehnica of Bucharest
\hfill\break\indent 060042 Bucharest, Romania.
\hfill\break\indent {\tt denisa.stancu@yahoo.com}}

\thanks{ $\,^{\star}$Corresponding author}




\date{June 14, 2026}
\maketitle

\begin{abstract}
{\footnotesize  For each open, bounded and convex domain $\Omega \subset \mathbb{R}^{D},$ $D\geq 2$, and each real number $p>1,$ we denote by $u_{p}$ the $p$\emph{-torsion function} on $\Omega $, i.e. the solution of the \emph{torsional creep problem} $\Delta_{p}u=-1$ in $\Omega $, $u=0$ on $\partial \Omega $, where $\Delta _{p}u:=\operatorname{div}( \left\vert \nabla u\right\vert ^{p-2}\nabla u) $ is the $p$-Laplacian. Let $T_p(\Omega)$ be the $p$\emph{-torsional rigidity} on $\Omega $, defined as $T_{p}\left( \Omega \right) :=\int_{\Omega }u_{p}dx$. Define $T\left( p;\Omega \right) :=\left\vert \Omega \right\vert ^{p-1}T_{p}\left( \Omega \right) ^{1-p}$, where $|\Omega|$ stands for the Lebesgue measure of $\Omega$. The main purpose of this paper is to compare the values of $T(p;\Omega)$ for bounded convex domains having different inradii. We prove that for any $0<a<b$ there exists an explicit constant $\gamma_{D,p}\in[1/D,1)$, depending only on the dimension $D$ and the parameter $p$, such that $T(p;\Omega_b)\leq T(p;\Omega_a)$, for all $ \Omega_a\in\PP^D(a)$, and $\Omega_b\in\PP^D(b)$, if and only if $\gamma_{D,p}b\geq a$, where $\PP^D(r)$ denotes the family of convex bounded domains in $\mathbb{R}^D$ of inradius $r$. In addition, we discuss the asymptotic equality case, the limiting regimes $p\rightarrow 1^+$ and $p\rightarrow\infty$, and the behaviour of the scale-invariant functional on model families such as rectangles, orthotopes, ellipses, and triangles. We also derive a fixed-volume monotonicity consequence for the unnormalised torsional rigidity, parameterized by the inradius.}
\end{abstract}

\section{Introduction and statement of main results}\label{sectia1}

\subsection{Definitions and notation}
For each integer $D\geq 1$ we denote by $\mathbb{R}^D$ the $D$-dimensional Euclidean space. 
For a bounded convex set $\Omega\subset\mathbb{R}^D$, denote by $R_\Omega$ its \emph{inradius}, i.e. the radius of the largest ball inscribed in $\Omega$, and by $\partial\Omega$ its boundary. 

Given $r>0$, we write $B_r^D$ for the open ball of radius $r$ in $\mathbb{R}^D$, and define the following two classes of sets
\[
\PP^D:=\{\Omega\subset\mathbb{R}^D \mid \Omega \text{ is open, bounded and convex}\},
\]
and
\[
\PP^D(r):=\{\Omega\in\PP^D \mid R_\Omega=r\}.
\]

For a given  $\Omega\in \PP^D$ and $p\in(1,\infty)$, we consider the $p$\emph{-torsion problem}
\begin{equation}\label{ptorsion}
\begin{cases}
-\Delta_{p}u=1 & \text{ in }\Omega, \\
u=0 & \text{ on }\partial \Omega,
\end{cases}
\end{equation}
where $\Delta_{p}u:=\mathrm{div}(|\nabla u|^{p-2}\nabla u)$ stands for the $p$-{\it Laplacian} operator. A simple application of the {\it Direct Method of the Calculus of Variations} shows that for each $p\in(1,\infty)$ problem (\ref{ptorsion}) possesses a unique solution, $u_{p}\in W_{0}^{1,p}(\Omega )$, called the $p$\emph{-torsion function} on $\Omega$. The associated $p$\emph{-torsional rigidity} on $\Omega$ is defined as follows
$$T_p(\Omega):=\int_\Omega u_p\;dx\,.$$
The following variational characterization will, however, be more useful for our subsequent analysis: 
$$T_p(\Omega)^{p-1}=\sup_{u\in W_0^{1,p}(\Omega)\setminus\{0\}}\frac{\left(\displaystyle\int_\Omega|u|\;dx\right)^p}{\displaystyle\int_\Omega|\nabla u|^p\;dx}\,,$$
(see, e.g.,  F. Della Pietra, N. Gavitone, \& S. Guarino Lo Bianco \cite[relations (18) and (19)]{DPGGLB}). Next, for convenience we introduce the \emph{normalized $p$-torsional rigidity}
\begin{equation}\label{defTnorm}
T(p;\Omega):=|\Omega|^{p-1} T_p(\Omega)^{1-p}, 
\end{equation}
where $|\Omega|$ is the Lebesgue measure of $\Omega$, which can also be written as
\begin{equation}\label{Tmin}
T(p;\Omega):=\inf_{u\in W_0^{1,p}(\Omega)\setminus\{0\}}\frac{\di\frac{1}{|\Omega|}\di\int_\Omega|\nabla u|^p\;dx}{\left(\di\frac{1}{|\Omega|}\di\int_\Omega|u|\;dx\right)^p}\,.
\end{equation}
Note that, for general domains $\Omega\in \PP^D$, neither $u_p$ nor $T(p;\Omega)$ can be computed explicitly.  

{In the engineering literature, torsional rigidity is a classical measure of the resistance of a prismatic bar to twisting. When $D=2$ and $p=2$, $\Omega$ represents the cross-section of the bar and the problem \eqref{ptorsion} reduces to the classical linear torsion problem, in which the torsional rigidity depends strongly on the geometry of the cross-section. The nonlinear case $p\neq2$ is also relevant in mechanics: it is related to torsional creep and to power-law type material response, in particular for bars subject to torsion over long time intervals at high temperature (see B. Kawohl \cite{K}; see also L.M. Kachanov \cite{KachanovCreep,KachanovPlasticity} for the underlying theory of creep and plasticity). In this setting, the parameter $p$ encodes the nonlinear constitutive behaviour of the material. Thus, comparison principles for $T_p(\Omega)$ and $T(p;\Omega)$ under geometric constraints such as prescribed inradius or fixed volume provide quantitative information on how the shape of the cross-section influences the torsional response, and are therefore relevant to shape optimization and engineering design.}

\subsection{Formulation of the problem}\label{subsect1.2}
For the ball $B_R^D$, the $p$-torsion function is explicitly known (see B. Kawohl \cite{K}):
$$
u_{p}(x)=\dfrac{D(p-1)}{p}\left[ \left( \dfrac{R}{D}\right) ^{\frac{p}{p-1}%
}-\left\vert \dfrac{x}{D}\right\vert ^{\frac{p}{p-1}}\right] ,\;\;\;\forall
\;x\in B_{R}^D\,.
$$
Further computations (see, e.g., \cite[relation (2.7)]{EM}) yield
$$T\left(p;B_R^D\right)=D\left(D+\frac{p}{p-1}\right)^{p-1}\frac{1}{R^p},\;\;\;\forall\;p>1\,.$$
Hence, for two different radii $0<R_1<R_2$, we have
$$T\left(p;B_{R_2}^D\right)<T\left(p;B_{R_1}^D\right),\;\;\;\forall\;p>1\,.$$
This observation naturally leads to the following question: \emph{given $\Omega_a\in\PP^D(a)$ and $\Omega_b\in\PP^D(b)$, with $0<a<b$ fixed, is it true that we have}
\begin{equation}\label{ineg}
T(p;\Omega_b)<T(p;\Omega_a),\qquad \forall p>1?
\end{equation}

\subsection{Motivation}
For comparison, recall that the {\it first Dirichlet eigenvalue of the $p$-Laplacian} on $\Omega\in \PP^D$ is defined by
\begin{equation}\label{dirichletmin}
\lambda_1(p;\Omega):=\min_{u\in W_0^{1,p}(\Omega)\setminus\{0\}}\frac{\di\int_\Omega|\nabla u|^p\;dx}{\di\int_\Omega|u|^p\;dx}\,.
\end{equation}
In \cite[Theorem 1]{MSD}, the last two authors established the following sharp comparison principle.

\begin{theorem}\label{mt1} 
Let $D\geq 1$ be a fixed integer. Given real numbers $p>1$ and $0<a<b$, we define the constant
\begin{equation}\label{constanta}
	C(D;p):=\left[\frac{\lambda_1\left(p;B^D_{1}\right)}{\lambda_1\left(p;B^1_{1}\right)}\right]^{1/p}>0\,.
\end{equation}
 We then have
\begin{equation}\label{complambda}
	\lambda_1(p;\Omega_b)< \lambda_1(p;\Omega_a),\;\;\;\forall\;\Omega_a\in\PP^D(a),\; \text{and}\;\Omega_b\in\PP^D(b)\,,
\end{equation}
if and only if $C(D;p)^{-1}b\geq a$.

Moreover, when $C(D;p)^{-1}b=a$ then there exists a sequence of domains $\{\Omega_n\}_n\subset\PP^D(a)$ for which the equality in (\ref{complambda}) holds  asymptotically if $\Omega_b=B^D_b$, in the sense that
  \begin{equation}\label{asycomplambda}
	\lim_{n\rightarrow\infty}\lambda_1(p;\Omega_n)=\lambda_1\left(p;B^D_b\right)\,.
\end{equation}
\end{theorem}

Since $\lambda_1(p;\Omega)$ and $T(p;\Omega)$ are defined through closely related
variational principles, it is natural to ask whether an analogue of Theorem \ref{mt1} holds for $T(p;\Omega)$. This expectation is further supported by the case $p=1$, where  both constants $\lambda_1(p;\Omega)$ and $T(p;\Omega)$ degenerate to the {\it Cheeger's constant} of the domain $\Omega$, denoted by $h(\Omega)$. In such a case, the following result is known, being also established in \cite[Corollary 1]{MSD}:

\begin{cor}\label{c1}
	Let $D\geq 1$ be a fixed integer  and $0<a<b$ be two given real numbers. Then 
\begin{equation}\label{comph}
	h(\Omega_b)\leq h(\Omega_a),\;\;\;\forall\;\Omega_a\in\PP^D(a),\;\;\Omega_b\in\PP^D(b)\,,
\end{equation}
if and only if $D^{-1}b\geq a$ (with strict inequality in (\ref{comph}) when $D^{-1}b>a$).

Moreover, when $D^{-1}b=a$ then there exists a sequence of domains $\{\Omega_n\}_n\subset\PP^D(a)$ for which the equality in (\ref{comph}) holds asymptotically if $\Omega_b=B^D_b$, in the sense that
	$$\lim_{n\rightarrow\infty}h(\Omega_n)=h\left(B^D_b\right)\,.$$
\end{cor}

\subsection{Main results}
We now state the main theorem of this paper, which provides the analogue of Theorem \ref{mt1} for the \emph{normalized $p$-torsional rigidity}.

\begin{theorem}\label{maintheorem} Let $D\geq 2$ be a fixed integer. Given real numbers $p>1$ and $0<a<b$, we define the constant
\begin{equation}\label{constantanoua}
	\gamma_{D,p}:=\left[\frac{T\left(p;B^D_{1}\right)}{T\left(p;B^1_{1}\right)}\right]^{-1/p}>0\,.
\end{equation}
We then have
\begin{equation}\label{compT}
	T\left(p;\Omega_b\right)\leq T\left(p;\Omega_a\right),\;\;\;\forall\;\Omega_a\in\PP^D(a),\;{\rm and}\;\Omega_b\in\PP^D(b)\,,
\end{equation}
if and only if $\gamma_{D,p}b\geq a$, with strict inequality in (\ref{compT}) when $\gamma_{D,p}b>a$.
\bigskip

Moreover, when $\gamma_{D,p}b=a$ then the equality in (\ref{compT}) occurs asymptotically if $\Omega_b=B^D_b$, in the sense that there exists a sequence of domains $\{\omega_n\}_n\subset\PP^D(a)$  for which we have
\begin{equation}\label{asycompT}
	\lim_{n\rightarrow\infty}T\left(p;\omega_n\right)=T\left(p;B^D_b\right)\,.
\end{equation}
\end{theorem}
\bigskip

As an immediate consequence, we obtain an answer to the problem we addressed in Subsection \ref{subsect1.2}.

\begin{cor}\label{maincor}
	Let $D\geq 2$ be a fixed integer and $0<a<b$ be two given real numbers. Then the following assertions hold.
	\smallskip
	
	If $D^{-1}b<a$ then there exists a real number $p_D\in(1,\infty)$ and two sets $\Omega_a\in\PP^D(a)$ and $\Omega_b\in\PP^D(b)$ such that
	\begin{equation}\label{inegnu}
	T\left(p_D;\Omega_a\right)<T\left(p_D;\Omega_b\right)\,.
    \end{equation}

    Conversely, if $D^{-1}b\geq a$ then for each set $\Omega_a\in\PP^D(a)$ and $\Omega_b\in\PP^D(b)$ we have
    \begin{equation}\label{inegda}
	T\left(p;\Omega_b\right)< T\left(p;\Omega_a\right),\;\;\;\forall\;p\in(1,\infty)\,.
    \end{equation}
\end{cor}


\subsection{Structure of the paper.} The rest of the paper is organized as follows.  In Section \ref{sectia2} we collect several basic properties of the normalized $p$-torsional rigidity.  Section \ref{sectia3} is devoted to the proofs of the main results.  In Section \ref{sectia4} we reformulate the comparison principle in terms of the average integral of the distance function to the boundary.  In Section \ref{sec:model-families} we evaluate the scale-invariant functional $Q_p(\Omega)$ defined in \eqref{defQ} on several model families of convex domains, including rectangles, higher-dimensional orthotopes, planar ellipses, and triangles.  In Section \ref{sec:limits} we discuss the limiting regimes as $p\to1^+$ and $p\to\infty$.  Finally, Section \ref{sectia7} contains concluding remarks and {a fixed-volume monotonicity consequence for the unnormalised torsional rigidity}.

\section{Properties of $T(p;\Omega)$}\label{sectia2} 

In this section we recall several known properties of $T(p;\Omega)$, when $D\geq 1$, $p\in(1,\infty)$ and $\Omega\in\PP^D$. These inequalities will be repeatedly used in the proofs of our main results. These include sharp lower and upper bounds of {\it Hersch-Protter} and {\it Buser type} inequalities, as well as a simple scaling rule.

\subsection{A lower bound of $T(p;\Omega)$}\label{subsection2.1}

According to F. Della Pietra, N. Gavitone, \& S. Guarino Lo Bianco \cite[Theorem 4.3]{DPGGLB} (see also F. Prinari \& A.C. Zagati \cite[Corollary 1.2]{PZ} applied with $q=1$) one has the following sharp {\it Hersch-Protter type inequality}
\begin{equation}\label{HPineq}
	T(p;\Omega)\geq\left(\frac{2p-1}{p-1}\right)^{p-1}\frac{1}{R_\Omega^p},\;\;\;\forall\;\Omega\in\PP^D,\;\forall\;p\in(1,\infty)\,.
\end{equation}
Note that $\left(\frac{2p-1}{p-1}\right)^{p-1}=T\left(p;B^1_{1}\right)$.
Moreover, we {recall} that by the proof of Corollary 1.2 (p. 13) from  \cite{PZ} (see also Subsections \ref{sectia5.1} and \ref{sectia5.2} below) we know that inequality (\ref{HPineq}) is {\it asymptotically sharp} by using, for example, the slab-type sequences $\Omega_L:=(-L,L)^{D-1}\times(0,2R)$ and observing that $\Omega_L\in\PP^D(R)$ (when $L>R$) and 
$$\lim_{L\rightarrow\infty}T\left(p;\Omega_L\right)=\left(\frac{2p-1}{p-1}\right)^{p-1}\frac{1}{R^p},\;\;\;\forall\;p\in(1,\infty)\,.$$

\subsection{An upper bound of $T(p;\Omega)$}

According to \cite[Theorem 1.1(ii)]{DPGGLB} (see also \cite{brasco2} for a similar result) one has the following sharp {\it Buser-type inequality}
\begin{equation}\label{buser}
	T\left(p;\Omega\right)\leq D\left(D+\tfrac{p}{p-1}\right)^{p-1}\frac{1}{R_\Omega^p},
\qquad \forall\,\Omega\in\PP^D,\;\forall\;p\in(1,\infty)\,.
\end{equation}
{Moreover, equality holds, up to translations and dilations, for balls.} More precisely, $D\Bigl(D+\tfrac{p}{p-1}\Bigr)^{p-1}=T\left(p;B^D_{1}\right)$.

\subsection{Rescaled domains}
Finally, here we note that if $t>0$ and $\Omega\in\PP^D$, then the rescaled domain $t\Omega:=\{tx: x\in\Omega\}$ satisfies (see, e.g., \cite[relation (2.5) from p. 5]{EM}) $R_{t\Omega}=tR_\Omega$, that is $t\Omega\in\PP^D(tR_\Omega)$, and 
\begin{equation}\label{rescaled}
T(p;t\Omega)={t^{-p}}T(p;\Omega),\;\;\;\forall\;p\in(1,\infty)\,.
\end{equation}
This scaling property will play a key role in the construction of extremal sequences in the proof of Theorem \ref{maintheorem}.

\section{Proof of the main results}\label{sectia3}
\subsection{Proof of Theorem \ref{maintheorem}}
Let  $\Omega\in\PP^D$.  We introduce the following scale-invariant quantity
\begin{equation}\label{defQ}
Q_p(\Omega):=\Bigg(\frac{T(p;\Omega)\,R_\Omega^p}{\big(\tfrac{2p-1}{p-1}\big)^{p-1}}\Bigg)^{1/p},
\end{equation}
and divide the proof into four steps.
\medskip

{\it  Step 1  (Choosing the constant $\gamma_{D,p}$).} 

{We first record a geometric corridor that will be used in Section~\ref{sec:model-families}. By \cite[Main Theorem]{brasco2} (applied with $q=1$) one has the Buser-type estimate
\[
T(p;\Omega)<\left(\frac{2p-1}{p-1}\right)^{p-1}\left(\frac{P(\Omega)}{|\Omega|}\right)^p,
\qquad \forall\,\Omega\in\PP^D,\ \forall\,p>1,
\]
where $P(\Omega)$ denotes the perimeter of $\Omega$, together with the geometric inequality (see \cite[Lemma B.1]{brasco2})
\[
\frac{R_\Omega}{D}\le\frac{|\Omega|}{P(\Omega)}<R_\Omega,
\qquad \forall\,\Omega\in\PP^D.
\]
Combining these with the Hersch--Protter inequality \eqref{HPineq} and recalling the definition \eqref{defQ} of $Q_p$, we obtain the geometric corridor
\begin{equation}\label{corridorQ}
1\;\leq\; Q_p(\Omega)\;<\;R_\Omega\,\frac{P(\Omega)}{|\Omega|}\;\leq\; D,
\qquad \forall\,\Omega\in\PP^D,\ \forall\,p>1\,.
\end{equation}
Combining instead Hersch-Protter inequality \eqref{HPineq},  Buser-type bound \eqref{buser}, and the facts that $\left(\frac{2p-1}{p-1}\right)^{p-1}=T\left(p;B^1_{1}\right)$ and $D\Bigl(D+\tfrac{p}{p-1}\Bigr)^{p-1}=T\left(p;B^D_{1}\right)$,
we obtain the sharp global corridor
\begin{equation}\label{corridorQsharp}
1\;\leq\; Q_p(\Omega)\;\leq\;\gamma_{D,p}^{-1},
\qquad \forall\,\Omega\in\PP^D,\ \forall\,p>1\,,
\end{equation}
where $\gamma_{D,p}$ is defined by relation (\ref{constantanoua}) and the upper endpoint is attained, up to translations and dilations, by balls.}

This motivates the definition of the following constants:
\begin{equation}\label{defalphabeta}
\alpha(p;D):=\inf_{\Omega\in\PP^D} Q_p(\Omega),\qquad
\beta(p;D):=\sup_{\Omega\in\PP^D} Q_p(\Omega).
\end{equation}
{ The lower endpoint of the corridor \eqref{corridorQ} is sharp. Indeed, the
Hersch--Protter type inequality \eqref{HPineq} is sharp, as recalled in
Subsection~\ref{subsection2.1} (see \cite[Theorem 4.3 and Proposition 4.4]{DPGGLB} and \cite[Corollary 1.2]{PZ}),
so that
\begin{equation}\label{alphaone}
\alpha(p;D)=1,\qquad \forall\,p>1,\ \forall\,D\ge 2.
\end{equation}
The upper endpoint, by contrast, is \emph{attained}. The result
\cite[Theorem 1.1(ii)]{DPGGLB} provides the corresponding sharp lower bound for
the scale-invariant torsional functional
$T_p(\Omega)\big/\bigl(|\Omega|\,R_\Omega^{\,p/(p-1)}\bigr)$, the minimum being
realized by the ball; via the definition \eqref{defTnorm} of $T(p;\Omega)$, this
is equivalent to the sharp upper bound for $T(p;\Omega)\,R_\Omega^{\,p}$.
Combining the explicit value
$T(p;B^D_R)=D\bigl(D+\tfrac{p}{p-1}\bigr)^{p-1}R^{-p}$ recalled in Subsection~\ref{subsect1.2}
with \cite[Theorem 1.1(ii)]{DPGGLB}, one obtains
\[
T(p;\Omega)\,R_\Omega^{\,p}\le D\Bigl(D+\tfrac{p}{p-1}\Bigr)^{p-1},
\qquad \forall\,\Omega\in\PP^D,
\]
with equality, up to translations and dilations, for balls. Recalling the
definitions \eqref{defQ} of $Q_p$ and (\ref{constantanoua}) of $\gamma_{D,p}$, this gives
\begin{equation}\label{betaball}
\beta(p;D)=Q_p\bigl(B^D_R\bigr)
=\left(\frac{p-1}{2p-1}\right)^{(p-1)/p}
\Bigl[\,D\Bigl(D+\tfrac{p}{p-1}\Bigr)^{p-1}\Bigr]^{1/p}=\gamma_{D,p}^{-1}\,,
\end{equation}
the supremum being attained, up to translations and dilations, for balls. In
particular, for $D=2$, $p=2$ one obtains $\beta(2;2)=\sqrt{8/3}$.}
Note that the use of infima and suprema reflects the fact that {the
infimum} is not attained within the class $\PP^D$, but {only}
approached by suitable degenerating sequences of convex domains,
{ whereas the supremum is attained by balls}.
We then have 
\begin{equation}\label{estimareab}
	1\,=\,\alpha(p;D)<\beta(p;D)=\gamma_{D,p}^{-1}\,.
\end{equation} 
In particular we note that 
$$\gamma_{D,p}=\frac{\alpha(p;D)}{\beta(p;D)}<1\,.$$
\medskip

{\it Step 2 (Direct implication).}  
Assume $\gamma_{D,p}b\geq a$.  We want to show that inequality (\ref{compT}) is valid. Suppose, for contradiction, that there exist domains $\Omega_a\in\PP^D(a)$ and $\Omega_b\in\PP^D(b)$ such that
\[
T(p;\Omega_a)<T(p;\Omega_b).
\]
Recalling the definition of $Q_p(\Omega)$ and combining this inequality with relations (\ref{HPineq}), (\ref{buser}) and (\ref{betaball}) we get
$$\left(\frac{2p-1}{p-1}\right)^{p-1}\frac{1}{a^p}\leq T(p;\Omega_a)< T(p;\Omega_b)\leq\left(\frac{2p-1}{p-1}\right)^{p-1}\frac{\gamma_{D,p}^{-p}}{b^p}\,.$$
This leads to the contradiction $\gamma_{D,p}b< a$.   Hence $T(p;\Omega_b)\leq T(p;\Omega_a)$ must hold.  

\medskip
Note that the above argument shows that for any two real numbers
$a,b\in(0,\infty)$ satisfying $\gamma_{D,p} b \ge a$, inequality
\eqref{compT} holds true. In particular, the assumption that $a$ and $b$
are fixed at the beginning of the proof of Theorem~\ref{maintheorem}
can be relaxed for this part of Step~2.
\medskip

Next, let us assume $\gamma_{D,p}b>a$. Let $\Omega_a\in\PP^D(a)$ and $\Omega_b\in\PP^D(b)$. Then, considering the rescaled domain $\tfrac{\gamma_{D,p}b}{a}\Omega_a\in\PP^D(\gamma_{D,p}b)$,  by the scaling property (\ref{rescaled}) we have
\begin{equation}\label{rescalarea}
	T\left(p;\frac{\gamma_{D,p}b}{a}\Omega_a\right)=\left(\frac{\gamma_{D,p}b}{a}\right)^{-p}T(p;\Omega_a)\,.
\end{equation}
Finally, since $\gamma_{D,p}\in(0,1)$ we have $a':=\gamma_{D,p}b<b$. 
By the remark above, relation \eqref{compT} holds for any pair of inradii satisfying 
$\gamma_{D,p}b\ge a'$, hence in particular for $(a',b)$. 
Rescale $\Omega_a$ to $\Omega_a':=\frac{a'}{a}\,\Omega_a\in\PP^D(a')$. 
Applying \eqref{compT} to $(\Omega_a',\Omega_b)$ and using \eqref{rescaled} yields
$$T(p;\Omega_b)\ \le\ T\!\left(p;\Omega_a'\right)
= \left(\frac{a'}{a}\right)^{-p}\,T(p;\Omega_a)
= \left(\frac{\gamma_{D,p}b}{a}\right)^{-p} T(p;\Omega_a).$$
{Since $\gamma_{D,p}b/a>1$, we have $(\gamma_{D,p}b/a)^{-p}<1$, and therefore the right-hand side is strictly smaller than $T(p;\Omega_a)$}, whence
$$T(p;\Omega_b) \;<\ T(p;\Omega_a).$$
Therefore, inequality (\ref{compT}) holds strictly.

\medskip
{\it  Step 3 (Converse implication).}  We assume that  relation (\ref{compT}) holds true, that is $T(p;\Omega_b)\leq T(p;\Omega_a)$ for all $\Omega_a\in\PP^D(a)$ and $\Omega_b\in\PP^D(b)$,  and we show that $\gamma_{D,p}b\geq a$.

Indeed, by the definition of $\alpha(p;D)$, there exists an extremal sequence of sets, $(\omega_n)_n\subset\PP^D$  such that
\begin{equation}\label{relatia1X}
\lim_{n\to\infty} Q_p(\omega_n)=\alpha(p;D)=1\,.
\end{equation}
Moreover, by the scaling invariance property (\ref{rescaled}), we recall that we have 
\[
T(p;t\Omega)R_{t\Omega}^p=T(p;\Omega)R_\Omega^p,\qquad \forall\,\Omega\in\PP^D,\;\forall\,t>0,
\]
Now, the limit in \eqref{relatia1X}  remains valid if we replace $\omega_n$ by the rescaled domain
$\tfrac{a}{R_{\omega_n}}\omega_n\in\PP^D(a)$. Hence we may assume $(\omega_n)_n\subset\PP^D(a)$  (since $R_{\frac{a}{R_{\omega_n}}\omega_n}=a$). In this case, relation (\ref{relatia1X}) becomes
\begin{equation}\label{relatia1}
	\lim_{n\rightarrow\infty}T(p;\omega_n)^{1/p}=\left(\frac{2p-1}{p-1}\right)^{(p-1)/p}\frac{1}{a}=T(p;B^1_1)^{1/p}\frac{1}{a}\,.
\end{equation}
On the other hand, since for each integer $n\geq 1$ we have $\omega_n\in\PP^D(a)$ and $B^D_b\in\PP^D(b)$, relation (\ref{compT}) gives
\begin{equation}\label{relatia3}
	T(p;B^D_b)\leq T(p;\omega_n),\;\;\;\forall\;n\geq 1\,,
\end{equation}
or, equivalently,
\begin{equation}\label{relatia4}
	T(p;B^D_b)^{1/p}\leq T(p;\omega_n)^{1/p},\;\;\;\forall\;n\geq 1\,.
\end{equation}
Letting $n\rightarrow\infty$ in (\ref{relatia4}) and using (\ref{relatia1}), we obtain
$$T(p;B^D_b)^{1/p}\;\leq\;T(p;B^1_1)^{1/p}\frac{1}{a}\,,$$
which in view of (\ref{constantanoua}) and (\ref{rescaled}) gives, $a\leq\gamma_{D,p}b$.
\medskip

{\it Step 4 (Equality case).}  We assume that $\gamma_{D,p}b=a$  and we show that the equality in (\ref{compT}) occurs asymptotically.

Indeed, as in the proof of Step 3, there exists a sequence of sets $(\omega_n)_n\subset\PP^D(a)$  such that relation (\ref{relatia1})  holds true. Since $\gamma_{D,p}b=a$ relation (\ref{relatia1}) reads, equivalently,     as follows
$$\lim_{n\rightarrow\infty}T(p;\omega_n)^{1/p}=\left(\frac{2p-1}{p-1}\right)^{(p-1)/p}\frac{1}{a}=T(p;B^1_1)^{1/p}\frac{1}{\gamma_{D,p}b}\,.$$
Now, recalling that by relation (\ref{constantanoua}) we have
$$\gamma_{D,p}:=\left[\frac{T\left(p;B^D_{1}\right)}{T\left(p;B^1_{1}\right)}\right]^{-1/p}\,,$$
the last two relations and (\ref{rescaled}) imply
$$\lim_{n\rightarrow\infty}T(p;\omega_n)^{1/p}=T(p;B^D_1)^{1/p}\frac{1}{b}=T(p;B^D_b)^{1/p}\,,$$
so equality in \eqref{compT} holds asymptotically, as claimed. The proof of Theorem \ref{maintheorem} is thus complete. \endproof 

\subsection{Proof of Corollary \ref{maincor}} For each $p\in(1,\infty)$ let $\gamma_{D,p}$ be defined by relation (\ref{constantanoua}). Simple computations show that
\begin{equation}\label{3iunie2026}
	\gamma_{D,p}=\left[\frac{1}{D}\left(\frac{2p-1}{{(D+1)p-D}}\right)^{p-1}\right]^{1/p}>\frac{1}{D},\;\;\;\forall\;p>1\,,
\end{equation}
and the lower bound $D^{-1}$ is {optimal}, since $\lim_{p\rightarrow 1^+}\gamma_{D,p}=D^{-1}$. Thus, we have 
$$\inf_{p\in(1,\infty)}\gamma_{D,p}=D^{-1}\,.$$
Assume that $D^{-1}b<a$. Then there exists $\varepsilon_0>0$ such that 
$(D^{-1}+\varepsilon_0)b<a$ (for instance, one may take 
$\varepsilon_0=\tfrac{1}{2}\big(\tfrac{a}{b}-\tfrac{1}{D}\big)$). 
Since $D^{-1}=\inf_{p\in(1,\infty)}\gamma_{D,p}$, it follows that there exists 
$p_{\varepsilon_0}\in(1,\infty)$ such that
$$\gamma_{D,p_{\varepsilon_0}}<D^{-1}+\varepsilon_0\,.$$
Consequently $\gamma_{D,p_{\varepsilon_0}}b<a$.  Applying Theorem \ref{maintheorem} with $p=p_{\varepsilon_0}$, it follows that there exist two sets $\Omega_a\in\PP^D(a)$ and $\Omega_b\in\PP^D(b)$ such that 
$$T\left(p_{\epsilon_0};\Omega_a\right)<T\left(p_{\epsilon_0};\Omega_b\right)\,,$$
i.e. inequality (\ref{inegnu}) holds true with $p_D=p_{\epsilon_0}$.
\medskip

On the other hand, if $D^{-1}b\geq a$, from relation (\ref{3iunie2026}), we have $\gamma_{D,p}b> D^{-1}b\geq a$, for all $p\in(1,\infty)$. By Theorem \ref{maintheorem} we deduce that (\ref{inegda}) holds.  The proof of Corollary \ref{maincor} is now complete.
 \endproof

\section{An extension in terms of the average distance function}\label{sectia4}

A key argument in the analysis of the problem considered in \cite{MSD} and in the formulation of Theorem 1 therein in terms of the inradius of the domain, is the following asymptotic formula due to P. Juutinen, P. Lindqvist, and J.J. Manfredi \cite[Lemma 1.5]{JLM99} (see also  N. Fukagai, M. Ito, and K. Narukawa \cite[Theorem 3.1]{FIN}):
\begin{equation}\label{asylambda}
\lim_{p\rightarrow\infty}\sqrt[p]{\lambda_1(p;\Omega)}=R_\Omega^{-1},\;\;\;\forall\;\Omega\in\PP^D\,.
\end{equation}
Denote by $\delta_\Omega$ the distance function to the boundary of $\Omega \in\PP^D$, i.e.
$$\delta_\Omega(x):=\inf_{y\in\partial\Omega}|x-y|,\qquad \forall x\in\Omega,$$
and define \emph{the average integral of the distance function} by
$$\delta(\Omega):=\frac{1}{|\Omega|}\int_\Omega \delta_\Omega(x)\,dx,\qquad \Omega\in\PP^D.$$
We recall that by \cite[relation (2.4)]{EM}, we know that we have 
\begin{equation}\label{asyT}
\lim_{p\rightarrow\infty}\sqrt[p]{T(p;\Omega)}=\delta(\Omega)^{-1},\;\;\;\forall\;\Omega\in\PP^D\,.
\end{equation}
For each real number $r>0$, we define
$$\PP^D[r]:=\{\Omega\in\PP^D \, | \; \delta(\Omega)=r\}\,.$$
Let now $0<a<b$ be two real numbers. Using (\ref{asyT}),  we observe that for any {\it fixed} domains $\Omega_a\in\PP^D[a]$ and  $\Omega_b\in\PP^D[b]$, there exists a constant $p_0=p_0(\Omega_a,\Omega_b)$ such that
\begin{equation}\label{Tstea}
	T(p;\Omega_b)<T(p;\Omega_a),\;\;\;\forall\;p>p_0\,.
\end{equation}
This naturally leads to a reformulation of the question raised at the end of Subsection \ref{subsect1.2}, namely: {\it is it true that, given $\Omega_a\in\PP^D[a]$ and $\Omega_b\in\PP^D[b]$, with $0<a<b$ fixed, inequality \eqref{Tstea} holds for every $p\in(1,\infty)$?}

\medskip
To investigate this question, we recall the following sharp inequality from L. Briani, G. Buttazzo, and F. Prinari \cite[Proposition 6.1]{BBP}
\begin{equation}\label{treisteleparalele}
	\frac{1}{D+1}\leq\frac{\delta(\Omega)}{R_\Omega}\leq\frac{1}{2},\;\;\;\forall\;\Omega\in\PP^D\,.
\end{equation}
Note that the value $(D+1)^{-1}$ from the left-hand side in the above inequality is achieved if $\Omega$ is a ball while the value $2^{-1}$ from the right-hand side in the above inequality is asymptotically attained by considering (for example) a sequence of slab domains $\Omega_\varepsilon:=(0,1)^{D-1}\times(0,\varepsilon)$ for $\varepsilon>0$, and letting $\varepsilon\rightarrow 0^+$ (see, the proof of Proposition 6.1, pp. 17-18, in \cite{BBP}).

Combining now \eqref{treisteleparalele} with \eqref{HPineq} and \eqref{buser}, we obtain
{\Small
\begin{equation}\label{patrusteleparalele}
\left(\frac{2p-1}{p-1}\right)^{p-1}
\left(\frac{1}{(D+1)\delta(\Omega)}\right)^p
\;\le\;
T(p;\Omega)
\;\leq\;
D\left(D+\frac{p}{p-1}\right)^{p-1}
\left(\frac{1}{2\,\delta(\Omega)}\right)^p,
\end{equation}} 
for all $\Omega\in\PP^D$ and all $p\in(1,\infty)$.

Moreover, since $\delta(t\Omega)=t\,\delta(\Omega)$ for all $t>0$, it follows that
$t\Omega\in\PP^D[t\,\delta(\Omega)]$, and hence scale-invariance remains valid when the functional is expressed
in terms of $\delta(\Omega)$ rather than $R_\Omega$. Consequently, for each $D\ge2$ and $p\in(1,\infty)$ we introduce the scale-invariant quantity
\begin{equation}\label{Qnou}
\overline Q_p(\Omega)
:=
\left[
\frac{T(p;\Omega)\,\delta(\Omega)^p}
{\big(\frac{2p-1}{p-1}\big)^{p-1}}
\right]^{1/p},
\qquad \Omega\in \PP^D,
\end{equation}
From \eqref{patrusteleparalele} and \eqref{Qnou} we immediately obtain the scale-free corridor
\begin{equation}\label{corridorQbar}
\frac{1}{D+1}\ \le\ \overline Q_p(\Omega)\ \le\ \frac{1}{2\gamma_{D,p}},
\qquad \forall\,\Omega\in\PP^D,\ \forall\,p\in(1,\infty)\,,
\end{equation}
where $\gamma_{D,p}$ is defined in relation (\ref{constantanoua}). This motivates the definition of
\begin{equation}\label{alphabetanou}
\overline \alpha(p;D)
:=\inf_{\Omega\in\PP^D}\overline Q_p(\Omega),
\qquad
\overline\beta(p;D)
:=\sup_{\Omega\in\PP^D}\overline Q_p(\Omega).
\end{equation}
We then set 
\begin{equation}\label{gamanou}
\overline\gamma_{D,p}
:=
\frac{\overline\alpha(p;D)}{\overline\beta(p;D)}.
\end{equation}
and deduce by (\ref{corridorQbar}) that $\overline\gamma_{D,p}\in\left[\frac{2\gamma_{D,p}}{D+1},1\right)$.
All these arguments allow us to repeat the proof of Theorem \ref{maintheorem}, with $R_\Omega$ replaced by $\delta(\Omega)$, and obtain the following extended result:

\begin{theorem}\label{mainteo2} Let $D\geq 2$ be a fixed integer. Let $p>1$ and $0<a<b$ be three given real numbers and define $\gamma_{D,p}$ as in relation (\ref{constantanoua}). Then there exists a constant $\overline\gamma_{D,p}\in \big[\,\tfrac{2\gamma_{D,p}}{D+1},1\big)$, depending only on $D$ and $p$, such that we have
\begin{equation}\label{compTnou}
	T\left(p;\Omega_b\right)\leq T\left(p;\Omega_a\right),\;\;\;\forall\;\Omega_a\in\PP^D[a],\;{\rm and}\;\Omega_b\in\PP^D[b]\,,
\end{equation}
if and only if $\overline\gamma_{D,p}b\geq a$, with strict inequality in (\ref{compTnou}) whenever $\overline\gamma_{D,p}b>a$.

Moreover, when $\overline\gamma_{D,p}b=a$ then the equality in (\ref{compTnou}) occurs asymptotically, in the sense that there exist two sequences of domains $\{\widetilde\omega_n\}_n\subset\PP^D[a]$ and $\{\widetilde\Omega_n\}_n\subset\PP^D[b]$ for which we have
\begin{equation}\label{asycompTnou}
	\lim_{n\rightarrow\infty}T\left(p;\widetilde\omega_n\right)= \lim_{n\rightarrow\infty}T\left(p;\widetilde\Omega_n\right)\,.
\end{equation}
\end{theorem}
\medskip
Next, we set
$$\overline\gamma_{D}:=\inf_{p\in(1,\infty)}\overline\gamma_{D,p}\,.
$$
Since for any $p>1$ we have $\overline\gamma_{D,p}\geq\frac{2\gamma_{D,p}}{D+1}$ and by (\ref{3iunie2026}) we know that $\gamma_{D,p}>D^{-1}$ it follows that $\overline\gamma_{D,p}>\frac{2}{D(D+1)}$, for all $p>1$. Thus, $\overline\gamma_{D}\geq\frac{2}{D(D+1)}$. Taking into account that fact, as a direct consequence of Theorem \ref{mainteo2}, we obtain the following corollary, which extends Corollary \ref{maincor} to the average-distance framework:
\medskip

\begin{cor}\label{maincor2}
	Let $D\geq 2$ be a fixed integer and $0<a<b$ be two given real numbers. Then there exists a constant $\overline\gamma_{D}\in \big[\,\tfrac{2}{D(D+1)},1\big)$, depending only on $D$, such that the following assertions hold.
	\smallskip
	
	If  $\overline\gamma_{D}b<a$ then there exists a real number $q_D\in(1,\infty)$ and two sets $\Omega_a\in\PP^D[a]$ and $\Omega_b\in\PP^D[b]$ such that
	\begin{equation}\label{inegnu1}
	T\left(q_D;\Omega_a\right)<T\left(q_D;\Omega_b\right)\,;
    \end{equation}
    
  Conversely, if $\overline\gamma_{D}b\geq a$ then for each set $\Omega_a\in\PP^D[a]$ and $\Omega_b\in\PP^D[b]$ we have
    \begin{equation}\label{inegda2}
	T\left(p;\Omega_b\right)\leq T\left(p;\Omega_a\right),\;\;\;\forall\;p\in(1,\infty)\,,
    \end{equation}
    with strict inequality in (\ref{inegda2}) when $\overline\gamma_{D}b>a$.
\end{cor}

\begin{remark}
The comparison constants associated with $\overline Q_p$ are also defined via
infima and suprema over $\PP^D$, and therefore the corresponding bounds should
be understood in a limiting sense. However, unlike the case of $Q_p$, the
available estimates for $\overline\gamma_{D,p}$  are not expected to be sharp.

In contrast, the explicit examples discussed in
Section~\ref{sec:model-families} {illustrate the behaviour of $Q_p$ within
several families of convex domains}. {For $Q_p$, the lower endpoint
$\alpha(p;D)=1$ is approached through degenerating families, whereas the upper
extremum $\beta(p;D)$ is attained by balls, see \eqref{betaball}.}
\end{remark}

\section{{Examples and behaviour on some families of convex sets}}
\label{sec:model-families}

In this section we evaluate the scale-invariant functional $Q_p(\Omega)$ on several model families of convex domains. {These examples illustrate the behavior of $Q_p$ within each family and the role that geometric degeneration may play, in some families, in approaching the lower endpoint $\alpha(p;D)=1$ of the corridor \eqref{corridorQ}. They should not, however, be interpreted as saturating the global upper endpoint in general: the global supremum $\beta(p;D)$ of $Q_p$ over $\PP^D$ is attained, up to translations and dilations, by balls, see \eqref{betaball}. Thus, in what follows, the family-wise estimates are based mainly on the geometric corridor \eqref{corridorQ}, while the sharp global corridor identifies the universal upper endpoint, which is attained, up to translations and dilations, by balls. Accordingly, the family bounds below are lower estimates for $\gamma_{\mathcal F}(p)$ rather than evaluations of the global constant $\gamma_{D,p}$.} Specifically, by highlighting these phenomena, we describe the behavior within each family $\mathcal F\subset \PP^D$, of the ratio
\[
\gamma_{\mathcal F}(p):=\frac{\alpha_{\mathcal F}(p)}{\beta_{\mathcal F}(p)},
\qquad
\alpha_{\mathcal F}(p):=\inf_{\Omega\in\mathcal F}Q_p(\Omega),
\qquad
\beta_{\mathcal F}(p):=\sup_{\Omega\in\mathcal F}Q_p(\Omega).
\]

\medskip
\noindent
In parallel with the average-distance framework of Section~\ref{sectia4}, we also comment on the qualitative behavior of the alternative scale-invariant quantity $\overline Q_p$ defined in \eqref{Qnou}. {While $Q_p$ yields explicit formulas or effective lower estimates for $\gamma_{\mathcal F}(p)$ in the model settings below}, for $\overline Q_p$ we only have universal (dimensional) bounds, and no comparable closed-form expressions for the associated constants are available.

\medskip
\noindent
Unless otherwise stated, all domains are open, bounded, and convex, and $R$ denotes the inradius, so that $\Omega\in\PP^D(R)$.

\subsection{Rectangles}\label{sectia5.1}
\label{subsec:rectangles}

For a given $\kappa \ge 2$, we consider the family of planar rectangles
\[
\mathcal{R}_\kappa := \big\{\, \Omega_{R,L}=(0,L)\times(0,2R) \ : \ L/R \ge \kappa \,\big\}.
\]
This family provides a simple setting in which one can track the behavior of $Q_p$
as the domain deforms from a square ($\kappa=2$) toward an infinitely long strip or line segment ($\kappa\to\infty$).

\begin{proposition}
For every $p>1$ and $\kappa \ge 2$, one has
\[
\gamma_{\mathcal{R}_\kappa}(p) \ge \frac{\kappa}{\kappa+2}.
\]
\end{proposition}

\begin{proof}
Let $\Omega=\Omega_{R,L}\in\mathcal{R}_\kappa$. A direct computation gives
\[
|\Omega|=2RL, \qquad P(\Omega)=2(2R+L), \qquad R_\Omega=R,
\]
and therefore
\[
R_\Omega \frac{P(\Omega)}{|\Omega|}
= 1+\frac{2}{L/R}
\le 1+\frac{2}{\kappa}.
\]
Recalling from \eqref{corridorQ} that
$1 \le Q_p(\Omega) < R_\Omega P(\Omega)/|\Omega|$,
we immediately obtain
\[
\gamma_{\mathcal{R}_\kappa}(p)
=\frac{\inf_{\mathcal{R}_\kappa} Q_p}{\sup_{\mathcal{R}_\kappa} Q_p}
\ge \frac{1}{1+2/\kappa}
= \frac{\kappa}{\kappa+2}.
\]
\end{proof}

\begin{remark}
For $\kappa=2$ (which includes the square) one obtains
$\gamma_{\mathcal{R}_2}(p)\ge 1/2$. Moreover, as $\kappa\to\infty$, the geometric corridor
\[
1 \le Q_p(\Omega) < 1+\frac{2}{\kappa}
\]
shrinks to the singleton $\{1\}$. Consequently,
$Q_p(\Omega)\to 1$ and $\gamma_{\mathcal R_\kappa}(p)\nearrow 1$ as $\kappa\to\infty$.
This conclusion is independent of whether the limit $\kappa\to\infty$ is achieved
by letting $L\to\infty$ with $R$ fixed or by letting $R\to 0$ with $L$ fixed, since both regimes force the same collapse of the geometric corridor.
\end{remark}

\begin{remark}[Analogue in terms of $\overline Q_p$]
Using the bounds in \eqref{treisteleparalele} and \eqref{patrusteleparalele}, one finds that 
\[
\frac{1}{3} \le \overline Q_p(\Omega_{R,L}) \le 1 ,
\]
independently of the aspect ratio $L/R$, with constants depending only on the dimension ($D=2$).

The behavior of $\overline Q_p(\Omega_{R,L})$ as $\kappa=L/R\to\infty$ depends on the way this limit is achieved. If $\kappa\to\infty$ through $L\to\infty$ with $R$ fixed, then the geometry converges
locally to an infinite strip and the $p$-torsion problem becomes asymptotically
independent of the longitudinal variable. In this regime one has $\overline Q_p(\Omega_{R,L}) \to 1/2$ (see, for instance, the computations from the proof of Proposition~6.1 in \cite[pp. 17-18]{BBP}).

If instead $\kappa\to\infty$ through the thin-domain regime $R\to0$ with $L$ fixed,
the distance function satisfies $0\le \delta _{\Omega_{R,L}}\le R$, and therefore
\[
0\le \delta(\Omega_{R,L})=\frac{1}{|\Omega_{R,L}|}\int_{\Omega_{R,L}} \delta _{\Omega_{R,L}}\,dx
\le R \longrightarrow 0 .
\]
Geometrically, the domain collapses to a one-dimensional segment.
In this regime we do not claim a specific limit for $\overline Q_p$,
although the same universal bounds still apply.
\end{remark}

\subsection{Higher-dimensional orthotopes}\label{sectia5.2}
\label{subsec:orthotopes}

For $\kappa\ge2$, consider the family of axis-aligned orthotopes
\[
\mathcal O_\kappa:=\Big\{\Omega=\prod_{i=1}^{D}(0,\ell_i):\
\ell_1=2R,\ \ell_j\ge\kappa R \text{ for } j=2,\dots,D \Big\}.
\]
Each element of $\mathcal O_\kappa$ is a $D$-dimensional axis-aligned box with
inradius $R$ and aspect ratio at least~$\kappa$ between the shortest side and
the remaining ones.

\begin{proposition}
For every $p>1$ and $\kappa\ge2$, we have
\[
\gamma_{\mathcal O_\kappa}(p)\ \ge\ \frac{\kappa}{\kappa+2(D-1)}.
\]
\end{proposition}

\begin{proof}
We have
\[
\frac{P(\Omega)}{|\Omega|}=2\sum_{i=1}^D\frac{1}{\ell_i}.
\]
With $\ell_1=2R$ and $\ell_j\ge\kappa R$ for $j\ge2$, it follows that
\[
R_\Omega\,\frac{P(\Omega)}{|\Omega|}
\ \le\ 2R_\Omega\!\left(\frac{1}{2R_\Omega}+\sum_{j=2}^{D}\frac{1}{\kappa R_\Omega}\right)
=1+\frac{2(D-1)}{\kappa}.
\]
Using \eqref{corridorQ} we deduce
\[
\gamma_{\mathcal O_\kappa}(p)
\ \ge\ \frac{1}{1+\tfrac{2(D-1)}{\kappa}}
=\frac{\kappa}{\kappa+2(D-1)}.
\]
\end{proof}

\begin{remark}
For $\kappa=2$ (which includes the hypercube) one obtains
$\gamma_{\mathcal O_\kappa}(p)\ge 1/D$, in agreement with the universal barrier
$[1/D,1)$.
Moreover, as $\kappa\to\infty$, the geometric corridor in \eqref{corridorQ}
shrinks to the singleton $\{1\}$. Consequently,
$Q_p(\Omega)\to 1$ and $\gamma_{\mathcal O_\kappa}(p)\nearrow 1$.
This conclusion is independent of whether the limit $\kappa\to\infty$ is achieved
by letting the sidelengths $\ell_j$ ($j\ge2$) tend to infinity with $R$ fixed,
or by letting $R\to 0$ with the other sidelengths fixed.
\end{remark}

\begin{remark}[Analogue in terms of $\overline Q_p$]
For any $\Omega\in\mathcal O_\kappa$, the bounds in \eqref{corridorQbar} yield
\[
\frac{1}{D+1} \le \overline Q_p(\Omega) \le \frac{D}{2},
\]
independently of the aspect ratio $\kappa$.
As a consequence,
$\overline\gamma_{\mathcal O_\kappa}(p)$ admits no sharper explicit expression than
the universal estimate $\overline\gamma_{D,p}\ge 2/(D(D+1))$.

The behavior of $\overline Q_p(\Omega)$ as $\kappa\to\infty$ depends on the mechanism
by which this limit is achieved.
If $\kappa\to\infty$ through the regime where the sidelengths
$\ell_j\to\infty$ for $j\ge2$ with $R$ fixed, the geometry converges locally to
$\mathbb R^{D-1}\times(0,2R)$ and the $p$-torsion problem becomes asymptotically
independent of the longitudinal variables. In this case,
$\overline Q_p(\Omega)\to 1/2$.

If instead $\kappa\to\infty$ through the thin-domain regime $R\to0$ with the other
sidelengths fixed, the inradius $R_\Omega$ tends to zero and the domain collapses
to a lower-dimensional set. Consequently, $\delta(\Omega)\to0$, and no specific
limit is claimed for $\overline Q_p$ in this regime, although the same universal bounds still apply.
\end{remark}

\subsection{Planar ellipses for $p=2$}
\label{subsec:ellipses}

For $\kappa\ge1$, we consider the family of ellipses
\[
\mathcal E_\kappa:=\big\{\,E=\{(x,y):x^2/a^2+y^2/b^2<1\},\ 1\le a/b\le\kappa\,\big\}.
\]
Note that for $E\in\mathcal E_\kappa$ with $a\ge b$, the inradius is given by $R_E=b$.

\begin{proposition}
For every $\kappa\ge1$, we have
\[
\gamma_{\mathcal E_\kappa}(2)
=\frac{1}{\sqrt2}\,\sqrt{1+\frac1{\kappa^2}}.
\]
\end{proposition}

\begin{proof}
First we note that for $E\in\mathcal E_\kappa$, the torsion function is explicit
(see, e.g., \cite{KMcN93}):
\[
u(x,y)=\frac{1-\frac{x^2}{a^2}-\frac{y^2}{b^2}}
{2\big(\frac1{a^2}+\frac1{b^2}\big)}.
\]
A direct computation yields
\[
T_2(E)=\frac{\pi a^3 b^3}{4(a^2+b^2)},\qquad
|E|=\pi ab,
\]
hence
\[
Q_2(E)
=\frac{2}{\sqrt3}\,\sqrt{1+\frac1{(a/b)^2}}.
\]
Since  $1\leq a/b\leq\kappa$, we have
\[
\alpha_{\mathcal E_\kappa}(2)
=\frac{2}{\sqrt3}\sqrt{1+\frac1{\kappa^2}},
\qquad
\beta_{\mathcal E_\kappa}(2)=\sqrt{\tfrac83},
\]
and the conclusion follows.
\end{proof}

\begin{remark}
The function $\gamma_{\mathcal E_\kappa}(2)$ is strictly decreasing in $\kappa$
and varies from $1$ (the disk $\kappa=1$) down to $1/\sqrt 2$ as
$\kappa\to\infty$.
\end{remark}

\begin{remark}[Analogue in terms of $\overline Q_2$]
For a planar ellipse $E\in\mathcal E_\kappa$ one has $R_E=b$ and
\[
\frac{b}{3} \le \delta(E) \le \frac{b}{2}.
\]
As a consequence,
\[
\frac13 \le \overline Q_2(E) \le \frac23,
\qquad
E\in\mathcal E_\kappa,
\]
and therefore $\overline\gamma_{\mathcal E_\kappa}(2)=1/2$.
Even in this completely explicit case, $\overline Q_2$ remains confined to a universal
interval and does not distinguish between the extremal shapes within the family,
in contrast with the sharp evaluation provided by $Q_2$.

In the thin-domain limit $b\to0$ with $a$ fixed, the inradius $R_E=b$ tends to
zero. By the geometric bounds relating $\delta(E)$ and $R_E$, the average
distance to the boundary also converges to zero.
The ellipse degenerates to a line segment, and in this regime
$\overline Q_2$ no longer distinguishes between different shapes within the
family, beyond the universal bounds.
\end{remark}

\subsection{Triangles}
\label{subsec:triangles}

We finally consider the family of all open, bounded triangles in the plane,
denoted by $\mathcal T\subset\PP^2$.

\begin{proposition}\label{triangleprop}
For every $p>1$, we have
\[
\gamma_{\mathcal T}(p)\ \ge\ \frac12 .
\]
\end{proposition}

\begin{proof}
For $T\in\mathcal T$, with inradius $R_T$, perimeter $P(T)$, and area $|T|$,
the well-known identity
\[
|T| = R_T\,\frac{P(T)}{2}
\]
implies
\[
R_T\,\frac{P(T)}{|T|}=2 .
\]
Inserted into the general corridor \eqref{corridorQ}, this yields
\[
1 \le Q_p(T) < 2 ,
\qquad \forall T\in\mathcal T ,\ \forall p>1 .
\]
Therefore,
\[
\gamma_{\mathcal T}(p)
= \frac{\inf_{\mathcal T} Q_p}{\sup_{\mathcal T} Q_p}
\ge \frac{1}{2}.
\]
\end{proof}

{\noindent
We stress that the upper bound $Q_p(T)<2$ used above comes from the geometric
corridor \eqref{corridorQ} and is not sharp on $\mathcal T$; hence the value
$1/2$ is only a lower estimate for $\gamma_{\mathcal T}(p)$ and need not be
optimal (see Remark~\ref{rem:triangles-sup} below).}

{%
\begin{remark}\label{rem:triangles-sup}
The estimate $Q_p(T)<2$ for triangles comes from the geometric corridor
\eqref{corridorQ} and is not sharp on $\mathcal T$. Indeed, by \eqref{betaball}
the global supremum of $Q_p$ over all planar convex sets is $\beta(p;2)<2$,
attained, up to translations and dilations, by disks; for $p=2$ one has
$\beta(2;2)=\sqrt{8/3}\approx 1.6329$. Consequently
$\sup_{T\in\mathcal T}Q_p(T)\le\beta(p;2)<2$, so the lower bound $1/2$ in
Proposition~\ref{triangleprop} need not be optimal.
\end{remark}}

\begin{remark}[Analogue in terms of $\overline Q_p$]
For every planar triangle $T$ one has the exact identity
\[
\delta(T)=\frac{R_T}{3}.
\]
Consequently,
\[
\frac13 \le \overline Q_p(T) \le \frac23 ,
\qquad T\in\mathcal T .
\]
Thus, even for triangles, where explicit geometric formulas are available,
$\overline Q_p$ remains confined to a fixed universal interval and does not
distinguish between extremal shapes within the family.
\end{remark}

\medskip
{Taken together, these examples clarify the geometric meaning of $\gamma_{D,p}$. The lower bound $Q_p(\Omega)\ge1$ is attained in dimension $D=1$; for $D\ge2$ it is the infimum of $Q_p$ over convex domains, approached in some families, for instance as $\kappa\to\infty$ for rectangles and orthotopes, through collapse to one-dimensional configurations. Within a fixed family, larger values of $Q_p$ tend to occur for more balanced geometries of smaller aspect ratio. These family-wise values, however, do not in general reach the global supremum $\beta(p;D)=\gamma_{D,p}^{-1}$, which over $\PP^D$ is attained by balls; for $D=2,p=2$, $\beta(2;2)=\sqrt{8/3}\approx1.6329$. Thus $\gamma_{D,p}$ is governed by the ball, while $\alpha(p;D)=1$ is approached only through degenerating families.}

In contrast, the scale-invariant quantity $\overline Q_p$ remains confined to
universal dimensional intervals in all these examples. In particular, when
degenerating families collapse to lower-dimensional sets, the inradius tends to
zero and, by the geometric inequalities linking $\delta(\Omega)$ and $R_\Omega$,
the average distance to the boundary vanishes as well. As a consequence,
$\overline Q_p$ does not capture the extremal geometric behavior responsible for
the sharpness of the comparison constants, even in settings where explicit
formulas are available.

\section{Limiting regimes: the cases $p\to1^+$ and $p\to\infty$}
\label{sec:limits}

The scale-invariant functionals $Q_p(\Omega)$ and $\overline Q_p(\Omega)$ provide
a natural bridge between the two extremal regimes of the $p$-torsion problem,
namely the Cheeger-type limit $p\to1^+$ and the distance-function-dominated
limit $p\to\infty$. In this section we recall the precise asymptotic behavior of
the normalized $p$-torsional rigidity $T(p;\Omega)$ and describe the
corresponding limits of $Q_p(\Omega)$ and $\overline Q_p(\Omega)$ within the
framework introduced in Sections~\ref{sectia1}--\ref{sectia4}.

For each $p>1$, let $u_p$ denote the solution of the $p$-torsion problem in
$\Omega$, and recall that
\[
T_p(\Omega):=\int_\Omega u_p\,dx,
\qquad
T(p;\Omega):=|\Omega|^{p-1}\,T_p(\Omega)^{\,1-p}.
\]
Recall also the definitions of the two scale-invariant functionals,
\[
Q_p(\Omega)^p
=
\frac{T(p;\Omega)\,R_\Omega^p}{\Big(\frac{2p-1}{p-1}\Big)^{p-1}},
\qquad
\overline Q_p(\Omega)^p
=
\frac{T(p;\Omega)\,\delta(\Omega)^p}{\Big(\frac{2p-1}{p-1}\Big)^{p-1}}.
\]

\subsection*{(i) The limit $p\to1^+$}

It is known (see, e.g. H. Bueno \& G. Ercole \cite[Theorem~2]{BE}) that
\[
\lim_{p\to1^+} T_p(\Omega)^{\,1-p}
=
h(\Omega),
\]
where
\[
h(\Omega):=\inf_{A\subset\Omega}\frac{P(A)}{|A|}
\]
denotes the Cheeger constant of\/ $\Omega$. Consequently,
\[
\lim_{p\to1^+} T(p;\Omega)=h(\Omega).
\]
Observe also that
\[
\lim_{p\to1^+}\Big(\tfrac{2p-1}{p-1}\Big)^{p-1}=1.
\]
Recalling the definitions above, we obtain
\[
\lim_{p\to1^+} Q_p(\Omega)
=
R_\Omega\,h(\Omega)
=: Q_1(\Omega),
\qquad
\lim_{p\to1^+} \overline Q_p(\Omega)
=\delta(\Omega)\,h(\Omega)
=: \overline Q_1(\Omega).
\]
Thus, in the limit $p\to1^+$, the functional $Q_p$ reduces to the
scale-invariant product $R_\Omega\,h(\Omega)$, while the functional
$\overline Q_p$ reduces to the product $\delta(\Omega)\,h(\Omega)$, where
$\delta(\Omega)$ denotes the average distance to the boundary. For convex
domains $\Omega\in\PP^D$, the functional $Q_1$ satisfies the corridor
\[
1
\le Q_1(\Omega)
\le D,
\qquad \forall\,\Omega\in\PP^D,
\]
obtained by letting $p\to1^+$ in the geometric corridor \eqref{corridorQ}. The lower endpoint is approached by slab-type boxes
\(\Omega_L=(-L,L)^{D-1}\times(-R,R)\), as \(L/R\to\infty\); in this regime
\(h(\Omega_L)R_{\Omega_L}\to1\), and hence \(Q_1(\Omega_L)\to1\), while the upper
endpoint is attained by balls, for which $h(B_R^D)=D/R$ and hence
$Q_1(B_R^D)=D$. For $\overline Q_1$, the same estimates yield the dimensional interval
\[
\frac{1}{D+1}\le \overline Q_1(\Omega)\le \frac{D}{2},
\qquad \forall\,\Omega\in\PP^D,
\]
obtained by combining the bounds
$\tfrac{1}{D+1}\le\delta(\Omega)/R_\Omega\le\tfrac12$ and
$1\le R_\Omega h(\Omega)\le D$; we do not claim sharpness of these endpoints.

\subsection*{(ii) The limit $p\to\infty$}

Recall from \eqref{asyT} that the normalized $p$-torsional rigidity satisfies
\[
\lim_{p\to\infty} T(p;\Omega)^{1/p} = \delta(\Omega)^{-1}.
\]
Observe also that
\[
\lim_{p\to\infty}\Big(\tfrac{2p-1}{p-1}\Big)^{(p-1)/p}=2.
\]
Recalling the definitions above, we obtain
\[
\lim_{p\to\infty} Q_p(\Omega)
=
\frac{R_\Omega}{2\delta(\Omega)}
=: Q_\infty(\Omega),
\qquad
\lim_{p\to\infty} \overline Q_p(\Omega)
=\frac12
=:\overline Q_\infty(\Omega).
\]
Thus, in the limit $p\to\infty$, the functional $Q_p$ reduces to the
scale-invariant ratio $R_\Omega/(2\delta(\Omega))$, while the functional
$\overline Q_p$ reduces to the constant value $1/2$, for every
$\Omega\in\PP^D$. By Proposition~6.1 in \cite{BBP}, the inradius $R_\Omega$ and
the average distance to the boundary $\delta(\Omega)$ are comparable through
purely dimensional constants, and the functional $Q_\infty$ satisfies the
corridor
\[
1
\le Q_\infty(\Omega)
\le \frac{D+1}{2},
\qquad \forall\,\Omega\in\PP^D.
\]
The lower endpoint is approached by slab-type boxes
\(\Omega_L=(-L,L)^{D-1}\times(-R,R)\), as \(L/R\to\infty\); indeed,
\(R_{\Omega_L}=R\), \(\delta(\Omega_L)/R_{\Omega_L}\to1/2\), and hence
\(Q_\infty(\Omega_L)\to1\), in
agreement with the asymptotic saturation observed in
Section~\ref{sec:model-families}. On the other hand, the upper endpoint is attained by balls,
for which $\delta(B_R^D)=R/(D+1)$ and hence $Q_\infty(B_R^D)=(D+1)/2$.

\begin{remark}
We emphasize that the present limiting discussion concerns both functionals
$Q_p(\Omega)$ and $\overline Q_p(\Omega)$, and that the two regimes are
symmetric in structure. In the regime $p\to1^+$, $Q_p(\Omega)$ converges to the
product $R_\Omega\,h(\Omega)$, thereby encoding simultaneously the inradius and
the Cheeger constant of the domain, while $\overline Q_p(\Omega)$ converges to
the product $\delta(\Omega)\,h(\Omega)$, encoding simultaneously the average
distance to the boundary and the Cheeger constant of the domain. In the opposite
regime $p\to\infty$, the functional $Q_p(\Omega)$ converges to the ratio
$R_\Omega/(2\delta(\Omega))$, linking the comparison principle to the
average-distance geometry discussed earlier, while the functional
$\overline Q_p(\Omega)$ converges to the constant value $\tfrac12$, for every
$\Omega\in\PP^D$.
\end{remark}

\section{Concluding remarks}\label{sectia7}

In this paper, we have established a unified comparison criterion for the
normalized $p$-torsional rigidity $T(p;\Omega)$ on convex domains, valid for
all $p\in(1,\infty)$.
More precisely, for any fixed dimension $D\ge 2$ and any pair of real numbers
$0<a<b$, we proved that there exists a constant
$\gamma_{D,p}$ given in relation (\ref{constantanoua}), depending only on $D$ and $p$, such that
\[
T(p;\Omega_b)\le T(p;\Omega_a),
\quad\text{for all }\Omega_a\in\PP^D(a),\ \Omega_b\in\PP^D(b),
\]
if and only if $\gamma_{D,p} b \ge a$.
This result provides the natural analogue, in the framework of
$p$-torsional rigidity, of classical comparison inequalities for the first
Dirichlet eigenvalue of the $p$-Laplacian (see \cite{MSD}).

{The introduction of the scale-invariant functional $Q_p(\Omega)$ allows for a precise geometric interpretation of the comparison principle established in this paper. The upper extremum $\beta(p;D)$ of $Q_p$ over $\PP^D$ is attained, up to translations and dilations, by balls, see \eqref{betaball}, and the comparison constant $\gamma_{D,p}=1/\beta(p;D)$ is therefore governed by the ball; the lower endpoint $\alpha(p;D)=1$ is not attained in dimension $D\ge2$ and is approached only through degenerating domains. The explicit evaluation of $Q_p(\Omega)$ on several model families, including rectangles, higher-dimensional orthotopes, planar ellipses, and triangles, illustrates this behaviour of $Q_p$ within these classes rather than the sharpness of the global upper endpoint: in several families, larger values of $Q_p$ are often associated with more balanced geometries, for instance squares among rectangles and hypercubes among orthotopes, whereas thin-domain regimes lead to values of $Q_p$ close to the lower endpoint $1$, reflecting collapse toward one-dimensional configurations.}

Throughout this work, the analysis has been carried out for the normalized
$p$-torsional rigidity $T(p;\Omega)$, rather than for the classical torsional
rigidity
\[
T_p(\Omega):=\int_\Omega u_p\,dx.
\]
This choice is dictated by the favorable scaling properties of $T(p;\Omega)$,
which are essential for the formulation and proof of the comparison principle
developed here.
Nevertheless, several explicit examples for convex domains suggest that
analogous comparison phenomena should also hold at the level of the unnormalized
functional $T_p(\Omega)$.
At present, however, the lack of an appropriate scale-invariant framework for
$T_p(\Omega)$ prevents a direct extension of the present approach, and the
identification of alternative normalization strategies or comparison mechanisms
leading to similar results remains an interesting direction for future research.  {Within this perspective, we point out that Theorem~\ref{maintheorem} yields a fixed-volume monotonicity statement for $T_p$, naturally related to but distinct from the classical Saint-Venant inequality.}  More precisely, the {\it Saint-Venant inequality} (see, e.g., \cite[Theorem 2.5]{DPG} for its general version, \cite[Chapter V]{polsze} and \cite[relation (4)]{vdBB} for its classical formulation obtained in the case where $p=2$, and the recent historical discussion in \cite{Berger}) states that for any open bounded set $\Omega\subset\mathbb{R}^D$ and any $p>1$, we have  
\begin{equation}\label{dSVineq}
	T_p(\Omega)\leq T_p(\Omega^\star)\,,
\end{equation}
where $\Omega^\star$ denotes a ball with $|\Omega^\star|=|\Omega|$. Next, for any real number $c>0$ let us define the set
$$\PP^D_c:={\{\Omega\in\PP^D:\;|\Omega|=c\}}\,.$$
Since the balls belonging to $\PP^D_c$ have the largest inradii (equal to $(c/v_D)^{1/D}$, where $v_D:=|B_1^D|$) among all the sets belonging to $\PP^D_c$, {and since they maximize $T_p$ within $\PP^D_c$ by the Saint-Venant inequality (\ref{dSVineq}), it is natural to ask whether $T_p$ is monotone with respect to the prescribed inradius inside the fixed-volume class, namely: \emph{given $c>0$, $p>1$ and $0<a<b$ four real numbers, is it true that we have}}
$$T_p(\Omega_a)<T_p(\Omega_b),\qquad \forall\; \Omega_a\in\PP^D_c\cap\PP^D(a),\;\Omega_b\in\PP^D_c\cap\PP^D(b)?$$
By Theorem \ref{maintheorem} and relation (\ref{defTnorm}) we get the following answer to the above question:
\begin{proposition}\label{dSVprop}
	Let $D\geq 2$ be a fixed integer. Let $p>1$, $c>0$ and $0<a<b$  be four given real numbers. {Assume that $b < (c/v_D)^{1/D}$, and let $\gamma_{D,p}$ be the constant given in relation (\ref{constantanoua}) from Theorem \ref{maintheorem}. If $\gamma_{D,p}b\geq a$, then}
\begin{equation}\label{inegdSV}
T_p(\Omega_a)\leq T_p(\Omega_b),\qquad \forall\; \Omega_a\in\PP^D_c\cap\PP^D(a),\;\Omega_b\in\PP^D_c\cap\PP^D(b)\,,
\end{equation}
{with strict inequality in (\ref{inegdSV}) whenever $\gamma_{D,p}b>a$.}
\end{proposition}

{
\begin{remark}
We state Proposition~\ref{dSVprop} as a sufficient condition only. Indeed,
Theorem~\ref{maintheorem} gives a necessary and sufficient condition on the full
classes $\PP^D(a)$ and $\PP^D(b)$, whereas Proposition~\ref{dSVprop} concerns the
restricted fixed-volume classes $\PP^D_c\cap\PP^D(a)$ and
$\PP^D_c\cap\PP^D(b)$. The extremal sequences used in the converse part of
Theorem~\ref{maintheorem} are rescaled to have prescribed inradii, and such
rescalings do not generally preserve the constraint $|\Omega|=c$. Therefore, an
intrinsic necessary and sufficient condition in the fixed-volume setting would
require constants defined directly on the restricted classes
$\PP^D_c\cap\PP^D(a)$ and $\PP^D_c\cap\PP^D(b)$, rather than the universal
constant $\gamma_{D,p}$. More explicitly, one could introduce the restricted
constants
\[
\alpha_{c,a}(p;D):=\inf_{\Omega\in\PP^D_c\cap\PP^D(a)}Q_p(\Omega),
\qquad
\beta_{c,b}(p;D):=\sup_{\Omega\in\PP^D_c\cap\PP^D(b)}Q_p(\Omega).
\]
These need not coincide with the global constants $\alpha(p;D)$ and $\beta(p;D)$
entering $\gamma_{D,p}=\alpha(p;D)/\beta(p;D)$, since the fixed-volume constraint
$|\Omega|=c$ restricts the admissible domains. Consequently, the condition
$\gamma_{D,p}b\ge a$ is sufficient for \eqref{inegdSV}, but it should not be
interpreted as necessary for the fixed-volume classes
$\PP^D_c\cap\PP^D(a)$ and $\PP^D_c\cap\PP^D(b)$.
\end{remark}}

Finally, we analyzed the limiting regimes $p\to1^+$ and $p\to\infty$.
In the former case, the comparison problem naturally connects to the Cheeger
constant and the $1$-Laplacian, while in the latter it is governed by the ratio
between the inradius and the average distance to the boundary, in the spirit of
$\infty$-Laplacian-type problems.
Possible directions for further research include the application of this scale-invariant framework to other variational functionals with similar homogeneity properties. While extensions to non-convex geometries present significant challenges due to the loss of sharp inradius-based estimates, the methodology developed here appears robust for studying anisotropic operators or for establishing quantitative stability estimates for the comparison inequalities.

\bigskip

{\bf Acknowledgments.} The authors would like to thank the anonymous referee for the careful reading of the initial version of the paper and for the insightful comments and suggestions which have led to substantial improvements.

\end{document}